\documentclass{article}[12pt]
\usepackage{amsmath}
\usepackage{amsthm}
\usepackage{amsfonts}
\usepackage{amssymb}
\usepackage{amstext}
\usepackage{amsopn}
\usepackage{latexsym}
\newtheorem{thm}{Theorem}[section]
\newtheorem{prob}{Open problem}

\newtheorem{lemma}{Lemma}[section]
\newtheorem{remark}{Remark}[section]

\begin{document}
\title{Periodic Beurling-Ahlfors Extensions and Quasisymmetric Rigidity of Carpets\thanks{Supported by NSFC (No. 1250010257).}}

\author{
Fan WEN \\ \\ Department of Mathematics
\\Jinan University
\\ Guangzhou 510632, China\\
wenfan@jnu.edu.cn}

\maketitle

\begin{abstract}
We establish periodic quasiconformal extension theorems for periodic orientation-preserving quasisymmetric self homeomorphisms of quasicircles or  quasi-round carpets. As applications, we prove that, if $f$ is a periodic orientation-preserving quasisymmetric self homeomorphism of a quasi-round carpet $S$ of measure zero in $\mathbb{C}$, which has a fixed point in the outer peripheral circle of $S$, then $f$ is the identity on $S$. Moreover, we prove that, if $f$ is a quasisymmetric self homeomorphism of a square carpet $S$ of measure zero in a rectangle ring, which fixes each of the four vertices of the outer peripheral circle of $S$, then $f$ is the identity on $S$. An analogous rigidity problem for the $\mathbb{C}^*$-square carpets is discussed.

\medskip


\medskip

\end{abstract}

\section{Introduction}

M. Bonk, B. Kleiner, and S. Merenkov established a nice quasiconformal geometry of carpets, by showing that quasisymmetric maps between carpets behave like conformal maps between domains in the complex plane $\mathbb{C}$; see \cite{B,B1,BKM,BM1,BM2}. The present paper addresses periodic quasiconformal extension and rigidity problem in this geometry and gives several key results and an open question.

We start with notation and basic facts that will be used later. Let $f:X\to Y$ be a topological homeomorphism of metric spaces $X$ and $Y$. We say that $f$ is quasisymmetric if there exists a homeomorphism $\eta:[0,\infty)\to[0,\infty)$ such that
$$d_Y(f(x), f(a))\leq\eta(t)d_Y(f(x),f(b))$$
for all triples of points $x, a, b$ in $X$ and $t\geq 0$ with $d_X(x,a)\leq td_X(x,b)$. In this case, we also say that $f$ is $\eta$-quasisymmetric. We say that $f$ is weakly quasisymmetric if there exists a constant $H\geq1$ such that
$$d_Y(f(x), f(a))\leq Hd_Y(f(x),f(b))$$
for all triples of points $x, a, b$ in $X$ with $d_X(x,a)\leq d_X(x,b)$; see Tukia-V\"{a}is\"{a}l\"{a} \cite{TV} or Heinonen \cite[Page 78]{H}. We say that $f$ is quasiconformal if
$$\sup_{x\in X}\limsup_{r\to 0}\frac{\sup_{a\in S(x,r)}d_Y(f(a), f(x))}{\inf_{b\in S(x,r)} d_Y(f(b),f(x))}<\infty,
$$
where $S(x,r)=\{a\in X: d_X(x,a)=r\}$.

It is known  that every $\eta$-quasisymmetric homeomorphism of $X$ and $Y$ can be extended to an $\eta$-quasisymmetric homeomorphism of their completions.

Quasisymmetry implies weak quasisymmetry. The inverse is true in some cases but not generally; see \cite[10.19]{H} and \cite[Theorem 6.6]{V1}. In particular,
every weakly quasisymmetric embedding of a connected subset of  $\mathbb{R}^n$ into $\mathbb{R}^m$ is quasisymmetric.

Quasisymmetry implies quasiconformality. The inverse is true in some cases but not generally; see \cite{Ty} and \cite[11.19]{H}. In particular,  every quasiconformal homeomorphism between two (closed) quasi-disks in $\mathbb{C}$ is quasisymmetric. Note also that a homeomorphism between domains $D$ and $G$  of $R^n$  is quasiconformal if and only if it is $\eta$-quasisymmetric in $B(x,d(x,\partial D)/2)$ for all $x\in D$, where $\eta$ is independent of $x$; see \cite[11.14]{H}. For the properties of quasiconformal maps we refer to V\"{a}is\"{a}l\"{a} \cite{V}.

For a subset $E$ of $\mathbb{C}$ we denote by $\overline{E}$, $\partial E$, $\mbox{int}(E)$, and $\mbox{diam}(E)$ its closure, boundary, interior, and diameter, respectively. Denote by $d(E,F)$ the distance between two subsets $E$ and $F$ of $\mathbb{C}$.
By a Jordan arc in $\mathbb{C}$, we mean a topological embedding of the unit interval into $\mathbb{C}$ or its image. By a Jordan closed path in $\mathbb{C}$, we mean a topological embedding of the unit circle into $\mathbb{C}$ or its image.
By a Jordan domain in $\mathbb{C}$, we mean a simply connected open subset of $\mathbb{C}$ whose boundary is a Jordan closed path. A Jordan closed path $\gamma$ in $\mathbb{C}$ is called a $c$-quasicircle for some constant $c\geq1$ if
$$\min\{\mbox{diam}(\alpha), \mbox{diam}(\beta)\} \leq c|z-w|$$
for every pair $z,w$ of distinct points in $\gamma$, where $\alpha$, $\beta$ are two distinct proper subarcs of $\gamma$ of endpoints $z$ and $w$. By Ahlfors \cite{A}, a Jordan closed path $\gamma$ in $\mathbb{C}$ is a quasicircle if and only if there is a quasiconformal map $f:\mathbb{C}\to \mathbb{C}$ such that $f(\gamma)$ is the unit circle. By a $c$-quasi-disk in $\mathbb{C}$, we mean a Jordan domain whose boundary is a $c$-quasicircle. Thus, a Jordan domain $D$ in $\mathbb{C}$ is a quasi-disk if and only if there is a quasiconformal map $f:\mathbb{C}\to \mathbb{C}$ such that $f(D)$ is the unit disk.

A subset of $\mathbb{C}$ is called a carpet if it is topologically homeomorphic to the standard $1/3$-Sierpi\'{n}ski carpet. According to Whyburn \cite{W}, a subset of $\mathbb{C}$ is a carpet if and only if it is compact, connected, locally connected, of topological dimension $1$, and has no local cut points. Moreover, every carpet $S$ in $\mathbb{C}$ can be written as
\begin{equation}\label{car}
S=D_0\setminus\bigcup_{k=1}^\infty D_k,
\end{equation}
where $D_0$ is a closed Jordan domain and $D_k$, $k\geq 1$, are Jordan domains in $\mathbb{C}$ satisfying the following conditions:

1) $\overline{D}_k\subset \mbox{int}(D_0)$ for each  $k\geq 1$,

2) $d(D_k,D_j)>0$ for each pair $k,j\geq 1$,

3) $\mbox{diam}(D_k)\to 0$ as $k\to \infty$, and

4) $\mbox{int}(S)=\emptyset$.

A Jordan closed path $\gamma$ in a carpet $S$ is called a peripheral circle of $S$ if $S\setminus \gamma$ is connected. Thus, for a carpet $S$ as in (\ref{car}), its peripheral circles are $\partial D_k$, $k\geq 0$, in which $\partial D_0$ is called its outer peripheral circle. By the definition, if $h: S\to T$ is a homeomorphism between two carpets $S$ and $T$ in $\mathbb{C}$, then $h$ maps the peripheral circles of $S$ into those of $T$.

For a carpet $S$ in $\mathbb{C}$, the bounded components of $\mathbb{C}\setminus S$ are simply connected in $\mathbb{C}$; however, the unbounded component is not.
Thus, if a homeomorphism $h: S\to T$ between two carpets $S$ and $T$ in $\mathbb{C}$ has a homeomorphic extension $H: \mathbb{C}\to \mathbb{C}$,  $h$ must preserve the outer peripheral circles.

By a quasi-round carpet, we mean a carpet in $\mathbb{C}$ whose peripheral circles are all $c$-quasicircles for some common constant $c\geq 1$.

According to Beurling-Ahlfors's extension theorem \cite{BA}, every quasisymmetric self homeomorphism of the unit circle has a quasiconformal extension to $\mathbb{C}$. As a consequence, if $\gamma_1$ and $\gamma_2$ are quasicircles in $\mathbb{C}$, then every quasisymmetric homeomorphism of $\gamma_1$ to $\gamma_2$ has a quasiconformal extension to $\mathbb{C}$.

We establish a periodic version of Beurling-Ahlfors's extension theorem, which does not follow directly from Beurling-Ahlfors's argument without any additional ingredient.

\begin{thm}\label{l3}
Let $\gamma$ be a quasicircle in $\mathbb{C}$. Then every $k$-periodic orientation-preserving quasisymmetric homeomorphism $f:\gamma\to \gamma$ has a $k$-periodic quasiconformal extension $F:\mathbb{C}\to \mathbb{C}$.
\end{thm}

Here $f$ is orientation-preserving, meaning that the images of any three points in clockwise order are also in clockwise order. A self map $f$ of a set is called periodic, if $f^k=id$ for some integer $k\geq 1$, where $f^k$ denotes the $k$-th iterate of $f$ and $id$ denotes the identity map. Notice that the homeomorphisms in Theorem \ref{l3} are a uncountable class.

Let $f:S\to T$ be a quasisymmetric homeomorphism between two quasi-round carpets $S$ and $T$ in $\mathbb{C}$. According to Bonk \cite[Proposition 5.1]{B}, $f$ has a quasiconformal extension $F:\widehat{\mathbb{C}}\to \widehat{\mathbb{C}}$. Moreover, $f$ has a quasiconformal extension $F:\mathbb{C}\to \mathbb{C}$ if and only if it preserves the outer peripheral circles of  $S$ and $T$.
We establish a periodic version of this extension result.

\begin{thm}\label{l2}
Let $S$ be a quasi-round carpet of measure zero in $\mathbb{C}$. Then every $k$-periodic orientation-preserving quasisymmetric homeomorphism $f:S\to S$ that preserves the outer peripheral circle of $S$ has a $k$-periodic quasiconformal extension $F:\mathbb{C}\to \mathbb{C}$.
\end{thm}

As an application of Theorem \ref{l2}, we have the following rigidity result.

\begin{thm}\label{c1} Let $S$ be a quasi-round carpet of measure zero in $\mathbb{C}$. If $f:S\to S$ is a periodic orientation-preserving quasisymmetric homeomorphism that has a fixed point in the outer peripheral circle of $S$, then $f=id$ on $S$.
\end{thm}

The peripheral circles of a carpet $S$ are called uniformly relatively separated if there is a constant $\delta>0$ such that $$\frac{d(C_1,C_2)}{\min\{\mbox{diam}(C_1), \mbox{diam}(C_2)\}}\geq \delta$$
for all pairs $C_1, C_2$ of peripheral circles of $S$ with $C_1\neq C_2$. Bonk-Merenkov \cite[Corollary 4.5]{BM1} proved that, if $S$ is a quasi-round carpet of measure zero in $\widehat{\mathbb{C}}$ whose peripheral circles are uniformly relatively separated, $C$ is a peripheral circle of $S$, and $p,q$ are two different points in $C$, then the group of all orientation-preserving quasisymmetric self homeomorphisms of $S$ fixing $p$ and $q$ is $\{id\}$ or an infinite cyclic group.
By contrast, Theorem \ref{c1} shows that the additional periodicity condition yields a stronger rigidity. For some recent works on quasisymmetric rigidity of metric carpets and self-similar quasi-round carpets we refer to \cite{R,S,Z}.


A carpet $S$ in a rectangle $R$ is called a square carpet if it can be written as
$$
S=R\setminus\bigcup_{k=1}^\infty Q_k,
$$
where $Q_k$, $k\geq 1$, are open squares of sides parallel to the coordinative axes. Let
$S$ and $T$ be respectively square carpets of measure zero in rectangles $[0,a]\times[0,1]$ and $[0,b]\times[0,1]$.
Bonk and Merenkov \cite[Theorem 1.4]{BM1} proved that, if $f:S\to T$ is an orientation-preserving quasisymmetric homeomorphism sending the four vertices of $[0,a]\times[0,1]$ to those of $[0,b]\times[0,1]$ and fixing the origin, then $a=b$,
$S=T$, and $f=id$ on $S$.

We say that $R\setminus  K$ is a rectangle ring if
$$
R=[0,a]\times [0,1]\ \mbox{ and }\ K=(s,s+w)\times (t,t+h)\ \mbox{ with }\ \overline{K}\subset\mbox{int}(R).
$$
A carpet $S$ in $R\setminus K$ is called a square carpet if it can be written as
$$
S=(R\setminus K)\setminus\bigcup_{k=1}^\infty Q_k,
$$ where $Q_k$, $k\geq 1$, are open squares as above. By the same argument as that of \cite[Theorem 1.4]{BM1}, we have the following lemma.

\begin{lemma}\label{1st lemma}
Let $S$ be a square carpet of measure zero in a rectangle ring $R\setminus K$. If $f: S\to S$ is a quasisymmetric homeomorphism fixing each of the four vertices of $R$ and satisfying $f(\partial K)=\partial K$ then $f=id$ on $S$.
\end{lemma}

We get the same result without assuming $f(\partial K)=\partial K$.

\begin{thm}\label{t1}
Let $S$ be a square carpet of measure zero in a rectangle ring $R\setminus K$ and $f: S\to S$ be a quasisymmetric homeomorphism fixing each of the four vertices of $R$. Then $f=id$ on $S$.
\end{thm}

To prove Theorem \ref{t1}, we first show that there is an integer $n\geq 1$ such that $f^n(\partial K)=\partial K$, which, combined with Lemma \ref{1st lemma}, implies that $f^n=id$ on $S$. This shows that every homeomorphism in Theorem \ref{t1} must be periodic, and hence we can apply Theorem \ref{c1} to obtain $f=id$ on $S$.

\medskip

Let $\mathbb{C}^*=\mathbb{C}\setminus\{0\}$. Define its line element and its area element by
$$
ds=\frac{|dz|}{|z|}\ \mbox{ and }\ d\sigma=\frac{dxdy}{|z|^2}.
$$
The metric given by the line element is known as the quasihyperbolic metric of $\mathbb{C}^*$; see \cite{GP}.
We say that $K$ is a $\mathbb{C}^*$-rectangle if
$$K=\{te^{i\theta}: a<t<b,\ \alpha<\theta<\beta\},$$
where $0<a<b<\infty$ and $0<\beta-\alpha<2\pi$. In this case, we say that $\log({b}/{a})$ and $\beta-\alpha$ are respectively its radial and circular directional  side-lengths. Moreover, if $\log({b}/{a})=\beta-\alpha$, we say that $K$ is a $\mathbb{C}^*$-square. A finite $\mathbb{C}^*$-cylinder means a set of the form $\{a\leq|z|\leq b\}$.

Let $A:=\{1\leq|z|\leq r\}$ be a finite $\mathbb{C}^*$-cylinder, where $r>1$. A carpet
$$
S:=A\setminus\bigcup_{k=1}^\infty Q_k
$$
is called a $\mathbb{C}^*$-square carpet in $A$ if $Q_k$, $k\geq 1$, are $\mathbb{C}^*$-squares. In this case, we say that $|z|=1$ and $|z|=r$ are respectively the inner and outer peripheral circles of $S$.

By Bonk and Merenkov \cite[Theorem 1.5]{BM1}, if $S$ is a $\mathbb{C}^*$-square carpet of measure zero in a $\mathbb{C}^*$-cylinder $A$ and $f$ is an orientation-preserving quasisymmetric self homeomorphism of $S$ that maps the inner and outer peripheral circles onto themselves respectively, then $f$ is the identity or a rational rotation.

Let $K$ be a $\mathbb{C}^*$-rectangle with $\overline{K}\subset\mbox{int}(A)$, where $A=\{1\leq|z|\leq r\}$. We say that a carpet
$$
S=(A\setminus K)\setminus\bigcup_{k=1}^\infty Q_k
$$
is a $\mathbb{C}^*$-square carpet in $A\setminus K$ if $Q_k$, $k\geq 1$, are $\mathbb{C}^*$-squares. By the same argument as that of \cite[Theorem 1.5]{BM1}, we have the following lemma.

\begin{lemma}\label{2cd lemma}
Let $S$ be a $\mathbb{C}^*$-square carpet of measure zero in  $A\setminus K$. Let $f$ be an orientation-preserving quasisymmetric self homeomorphism of $S$ that maps $\partial K$, the inner, and the outer peripheral circles of $S$ onto themselves, respectively. Then $f=id$ on $S$.
\end{lemma}

We get the following result without assuming $f(\partial K)=\partial K$.
\begin{thm}\label{t1.6}
Let $S$ be a $\mathbb{C}^*$-square carpet of measure zero in  $A\setminus K$. Let $f$ be an orientation-preserving quasisymmetric self homeomorphism of $S$ that maps the inner and outer peripheral circles of $S$ onto themselves, respectively. Then $f$ is periodic.
\end{thm}

\begin{prob}
If the $\mathbb{C}^*$-rectangle $K$ in Theorem \ref{t1.6} is not a $\mathbb{C}^*$-square, we conjecture that the conditions of this theorem force $f$ to be the identity on $S$, but a proof remains elusive.
\end{prob}

The paper is organized as follows. We prove Theorems \ref{l3} and \ref{l2} in Section 2. In Section 3, we establish a topological rigidity result for periodic orientation-preserving self homeomorphisms of Jordan closed domains (Lemma \ref{l1}), which together with Theorem \ref{l2} yields Theorem \ref{c1}. Finally, we prove
Theorems \ref{t1} and \ref{t1.6} in Section 4.

\section{Periodic Quasiconformal Extensions}

In this section, we prove Theorems \ref{l3} and \ref{l2}.

\medskip

\noindent{\bf Proof of Theorem \ref{l3}.} Let $\gamma$ be a quasicircle in $\mathbb{C}$ and let $f:\gamma\to \gamma$ be a $k$-periodic orientation-preserving quasisymmetric homeomorphism. Let $g:\mathbb{C}\to \mathbb{C}$ be a quasiconformal homeomorphism such that $g(\gamma)$ is the unit circle $ S^1 $. Then $gfg^{-1}: S^1 \to  S^1 $ is a $k$-periodic orientation-preserving quasisymmetric homeomorphism. If $gfg^{-1}$  has a
$k$-periodic quasiconformal extension $F:\mathbb{C}\to \mathbb{C}$ then $g^{-1}Fg$ is a $k$-periodic quasiconformal extension of $f$.

Therefore we only need to show that every $k$-periodic orientation-preserving quasisymmetric homeomorphism $f: S^1 \to  S^1 $ has a $k$-periodic quasiconformal extension $F:\mathbb{C}\to \mathbb{C}$.
To verify this, we adopt a strategy as follows: First, we show that, for each $r\in(0,1)$, $f$ has a $k$-periodic quasisymmetric extension $\widetilde{f}: A_r\to A_r$, where $$A_r=\{r\leq|z|\leq1\}.$$ Then we show that $\widetilde{f}$ has a $k$-periodic quasiconformal extension  $F:\mathbb{C}\to \mathbb{C}$ by repeatedly applying the Schwarz reflection principle.

In what follows, by $A\asymp B$ we mean $C^{-1}B\leq A\leq CB$ and by $A\preceq B$ we mean $A\leq CB$ for some constant $C\geq 1$. The proof is completed by the following three lemmas.

\begin{lemma}\label{l2.3}
Let $\gamma_1,\gamma_2$ be quasicircles in $\mathbb{C}$,  $h:\gamma_1\to \gamma_2$ be a homeomorphism, and $a,b,c,d\in \gamma_1$ be four distinct points  in clockwise (or counterclockwise) order. Suppose that the restriction of $h$ to each of the four proper subarcs $ab$ $bc$, $cd$, and $da$ of $\gamma_1$ is quasisymmetric and that, for every point $v$ in $\{a,b,c,d\}$, there exists a constant $M\geq 1$ such that
$$
|h(v)-h(s)|\leq M|h(v)-h(t)|
$$
for all $s, t\in\gamma_1$ with $|v-s|\leq|v-t|$. Then $h$ is quasisymmetric.
\end{lemma}

\begin{proof}
In the case  $\gamma_1=\gamma_2= S^1 $, a direct computation by using the assumptions shows that $h$ is weakly quasisymmetric, which is left to the readers. Since every weakly quasisymmetric self homeomorphism of $ S^1 $ is quasisymmetric, we conclude that $h$ is quasisymmetric.

In the general case, let $\gamma_1:  S^1 \to \gamma_1$ and $\gamma_2:  S^1 \to \gamma_2$ be respectively quasisymmetric parameterizations of $\gamma_1$ and $\gamma_2$. Then we easily check that  $\gamma_2^{-1} h \gamma_1: S^1 \to S^1 $ satisfies the conditions
of the theorem for the points $\gamma_1^{-1}(a)$, $\gamma_1^{-1}(b)$, $\gamma_1^{-1}(c)$, and $\gamma_1^{-1}(d)$. Thus $\gamma_2^{-1}h\gamma_1$ is quasisymmetric by applying the conclusion of the special case. It then follows that $h$ is quasisymmetric. Here we abuse notation by using the same symbol for a quasicircle and its parameter representation.
\end{proof}

The following lemma is a key ingredient of the proof.

\begin{lemma}\label{l2.5}
Let $A_r:=\{r\leq|z|\leq1\}$ be a closed annulus, where $0<r<1$. Then every $k$-periodic orientation-preserving quasisymmetric homeomorphism $f: S^1 \to  S^1 $ has a $k$-periodic quasiconformal extension $\widetilde{f}: A_r\to A_r$.
\end{lemma}
\begin{proof}
Without loss of generality assume that $k$ is the smallest period of $f$ and that $k\geq 2$.

Let $u_j=f^j(u_0)$ and let $l_j$ be the straight line segment joining $ru_j$ and $u_j$, where $u_0$ is an arbitrarily given point on $ S^1 $ and $j=0,1,\cdots, k+1$. Since $f$ is orientation-preserving, we see that $u_0,u_1,\cdots, u_{k-1}$ are distinct points on $ S^1 $ in the clockwise or counterclockwise direction, $u_k=u_0$, $l_k=l_0$, $u_{k+1}=u_1$, and $l_{k+1}=l_1$.

For each $j\leq k$, let $\alpha_j$ be the subarc of $ S^1 $ from $u_{j}$ to $u_{j+1}$ in the same direction as above. Then the interiors of $\alpha_0,\alpha_1,\cdots, \alpha_{k-1}$ in $ S^1 $ are pairwise disjoint and $\alpha_k=\alpha_0$. Let $$C_j=l_{j}\cup\alpha_j\cup l_{j+1}\cup r\alpha_j,$$ where $r\alpha_j=\{rz: z\in\alpha_j\}$. Then $C_j$ is a quasicircle in $A_r$ and $C_k=C_0$.
Let $D_j$ be the closed quasi-disk in $A_r$ bounded by $C_j$. Then the interiors of $D_0,D_1,\cdots, D_{k-1}$ are pairwise disjoint, $D_k=D_0$, and
\begin{equation}\label{d1}
A_r=\bigcup_{j=0}^{k-1}D_j.
\end{equation}

For each $j<k$ let $p_j:l_j\to\alpha_j$ be a bilipschitz map such that
$$
p_j(u_j)=u_j\ \mbox{ and }\ p_j(ru_j)=u_{j+1}.
$$
Let $p_{k}=p_0$ and $p_{k+1}=p_1$. Then $p_{j+1}^{-1} f p_j:l_j\to l_{j+1}$ is a quasisymmetric homeomorphism
satisfying
$$
p_{j+1}^{-1} f p_j(u_j)=u_{j+1}\ \mbox{ and }\ p_{j+1}^{-1} f p_j(ru_j)=ru_{j+1}.
$$

For each $j<k$ let $q_j:r\alpha_j\to\alpha_{j+1}$ be a bilipschitz map such that
$$
q_j(ru_j)=u_{j+1}\ \mbox{ and }\ q_j(ru_{j+1})=u_{j+2}.
$$
Let $q_{k}=q_0$. Then $q_{j+1}^{-1} f q_j:r\alpha_j\to r\alpha_{j+1}$ is a quasisymmetric homeomorphism
satisfying
$$
q_{j+1}^{-1} f q_j(ru_j)=ru_{j+1}\, \mbox{ and }\, q_{j+1}^{-1} f q_j(ru_{j+1})=ru_{j+2}.
$$

For each $j<k$ let $h_j:C_j\to C_{j+1}$ be defined by
$$h_j(z)=\left\{
\begin{array}{llll}
p_{j+1}^{-1} f p_j(z), & \hbox{if $z\in l_j$} \\ \\
f(z), & \hbox{if $z\in \alpha_j$} \\ \\
p_{j+2}^{-1} f p_{j+1}(z), & \hbox{if $z\in l_{j+1}$}\\ \\
q_{j+1}^{-1} f q_j(z), & \hbox{if $z\in r\alpha_j$.}
\end{array}
\right.$$
Let $h_k=h_0$. By the construction, $h_j$ is a well-defined orientation-preserving homeomorphism that maps the points $ru_j$, $u_j$, $u_{j+1}$, and $ru_{j+1}$ to the points $ru_{j+1}$, $u_{j+1}$, $u_{j+2}$, and $ru_{j+2}$, respectively. Moreover, the restriction of $h_j$ to each of $l_j$, $\alpha_j$, $l_{j+1}$, and $r\alpha_j$ is quasisymmetric. In addition, we have
\begin{equation}\label{b1}
h_j|_{\alpha_j}=f|_{\alpha_j}\ \mbox{ and }\ h_j|_{l_{j+1}}=h_{j+1}|_{l_{j+1}}.
\end{equation}

We claim that
\begin{equation}\label{b2}
h_{k-1}\cdots h_{1}h_0=id \mbox{ on }C_0.
\end{equation}
In fact, if $z\in \alpha_0$, then
$$h_{k-1}\cdots h_{1}h_0(z)=f^k(z)=z.$$
If $z\in l_0$, then
$$h_{k-1}\cdots h_{1}h_0(z)=(p_{k}^{-1} f p_{k-1})\cdots(p_2^{-1} f p_1)(p_1^{-1} f p_0)(z)=p_0^{-1}f^k p_0(z)=z.$$
It is equally straightforward to show that $h_{k-1}\cdots h_{1}h_0=id$ on $l_1$ and $r\alpha_0$.

In this context, we see that, for a given integer $0\leq j<k$, to show that $h_j$ is quasisymmrtric by using Lemma \ref{l2.3}, it suffices to prove that, for every point $v\in\{u_j,u_{j+1}, ru_{j+1}, ru_j\}$, there is a constant $M\geq 1$ such that
\begin{equation}\label{n3}
|h_j(v)-h_j(s)|\leq M|h_j(v)-h_j(t)|
\end{equation}
for all $s, t\in C_j$ with $|v-s|\leq|v-t|$.

We first consider the point $u_j$. Let $s, t\in C_j$ satisfy $|u_j-s|\leq|u_j-t|$. We are going to show
\begin{equation}\label{a1}
|h_j(u_j)-h_j(s)|\preceq|h_j(u_j)-h_j(t),
\end{equation}
where the hidden constant is independent of $s,t$.

If $t\in l_{j+1}\cup r\alpha_j$, one has $h_j(t)\in l_{j+2}\cup r\alpha_{j+1}$, and hence
$$0<d(u_{j+1}, l_{j+2}\cup r\alpha_{j+1})\leq|u_{j+1}-h_j(t)|=|h_j(u_j)-h_j(t)|.$$ Then the inequality (\ref{a1}) follows by
\begin{eqnarray*}
|h_j(u_j)-h_j(s)|&\leq&\mbox{diam}(C_{j+1})\\ &\leq&\frac{\mbox{diam}(C_{j+1})}{d(u_{j+1}, l_{j+2}\cup r\alpha_{j+1})}|h_j(u_j)-h_j(t)|
\end{eqnarray*}
In what follows we assume that $t\in l_{j}\cup \alpha_j$.

If $s,t\in l_j$ or $s,t\in \alpha_j$, the inequality (\ref{a1}) follows by the quasisymmetry of $h_j|_{l_j}$ or $h_j|_{\alpha_j}$, respectively.

If $s\in l_j$ and $t\in\alpha_j$, one has by $p_j(u_j)=u_j$ and the Lipschitz property of $p_j$
$$|u_j- p_j(s)|=|p_j(u_j)- p_j(s)|\preceq|u_j- s|\leq|u_j-t|.$$
Then, by the definition of $h_j$, the bilipschitz property of $p_{j+1}$, the quasisymmetry of $f$, $f(u_j)=h_j(u_j)$, and $f(t)=h_j(t)$, we get
\begin{eqnarray*}
|h_j(u_j)-h_j(a)|&=&|p_{j+1}^{-1}f p_j(u_j)- p_{j+1}^{-1}f p_j(a)|\\
&\asymp& |f(u_j)- fp_j(s)| \\
&\preceq& |f(u_j)- f(t)| \\
&=& |h_j(u_j)-h_j(t)|.
\end{eqnarray*}

If $s\in \alpha_j$ and $t\in l_j$, one has $h_j(s)=f(s)$ and $$|u_j-s|\leq|u_j-t|\preceq|p_j(u_j)-p_j(t)|=|u_j-p_j(t)|.$$ Thus
\begin{eqnarray*}
|h_j(u_j)-h_j(s)| &=& |f(u_j)-f(s)| \\
&\preceq& |f(u_j)-fp_j(t)| \\
&=& |fp_j(u_j)-fp_j(t)| \\
&\asymp& |p_{j+1}^{-1}f p_j(u_j)- p_{j+1}^{-1}f p_j(t)| \\
&=&|h_j(u_j)-h_j(t)|.
\end{eqnarray*}

If $s\in l_{j+1}\cup r\alpha_j$,
then, by the assumptions, one has  $$d(u_j, l_{j+1}\cup r\alpha_j)\leq|u_j-s|\leq|u_j-t|.$$
From the construction of $C_j$ we see that there are exactly two points $w_j$ in $l_{j}\cup \alpha_j$ satisfying $$|u_j-w_j|=d(u_j, l_{j+1}\cup r\alpha_j).$$ We choose one such $w_j$ arbitrarily. Then $|u_j-w_j|\leq|u_j-t|$. Since $t\in l_{j}\cup \alpha_j$ is assumed, we have by the previous conclusion that
$$|h_j(u_j)-h_j(w_j)|\preceq|h_j(u_j)-h_j(t)|.$$
Therefore
$$|h_j(u_j)-h_j(s)|\preceq \mbox{diam}(C_{j+1})\preceq \frac{\mbox{diam}(C_{j+1})}{|h_j(u_j)-h_j(w_j)|}|h_j(u_j)-h_j(t)|.$$

We have considered all cases, and the inequality (\ref{a1}) holds in each one with a hidden constant independent of $s,t$. This eastablishes  the inequality (\ref{n3}) for $v=u_j$ with a constant $M\geq 1$ independent of $s,t$. By an analogous argument, one may prove the same conclusion for $v=u_{j+1}, ru_{j+1}, ru_j$. Therefore $h_j$ is quasisymmetric by Lemma \ref{l2.3}.

Now  it follows from Beurling-Ahlfors theorem that $h_j:C_j\to C_{j+1}$ has a quasiconformal extension $H_j:\mathbb{C}\to \mathbb{C}$ such that $H_j(D_j)=D_{j+1}$ for every integer $0\leq j\leq k-2$. Once these extensions are chosen, we define
$$H_{k-1}=H_0^{-1}H_1^{-1}\cdots H_{k-2}^{-1}.$$
Then $H_{k-1}$ is a quasiconformal extension of $h_{k-1}$ by (\ref{b2}), and we have
\begin{equation}\label{b3}
H_{k-1}H_{k-2}\cdots H_{1}H_0=id\,\mbox{ on }\,D_0,
\end{equation}

Let $\widetilde{f}:A_r\to A_r$ be defined by
$$\widetilde{f}(z)=\left\{
\begin{array}{llll}
H_0(z), & \hbox{if $z\in D_0$} \\ \\
H_1(z), & \hbox{if $z\in D_1$} \\ \\
\cdots & \\ \\
H_{k-1}(z), & \hbox{if $z\in D_{k-1}$.}
\end{array}
\right.$$
By using (\ref{d1}), (\ref{b1}), and (\ref{b3}), we easily check that $\widetilde{f}$ is a well-defined $k$-periodic homeomorphic extension of $f$.
Moreover, since $H_j$'s are quasiconformal, we see that $\widetilde{f}$ is quasiconformal on $\cup_{j=0}^{k-1}\mbox{int}(D_j)$. Thus $\widetilde{f}$ is quasiconformal on $\mbox{int}(A_r)$  by the removability of quasiconformality, and hence quasisymmetric on $A_r$.
\end{proof}

\begin{lemma}\label{l2.4}
Let $r\in(0,1)$, $A_r=\{r\leq|z|\leq1\}$, and $f: A_r\to A_r$ be a $k$-periodic quasisymmetric homeomorphism with $f( S^1 )= S^1 $. Then $f$ has a $k$-periodic quasiconformal extension $F: \mathbb{C}\to \mathbb{C}$.
\end{lemma}
\begin{proof}
Let $S_r$ denote the circle $|z|=r$ and let
$$R_r(z)=\frac{r^2}{\overline{z}}$$ be the reflection through $S_r$. Then we have $$\mbox{$R_r^2=id$, $R_r(z)=z$ for $z\in S_r$, and $R_r( S^1 )=S_{r^2}$}.$$

Let $f_0=f$. Let $f_1:A_{r^2}\to A_{r^2}$ be defined by
$$f_1(z)=\left\{
\begin{array}{lll}
f_0(z), & \hbox{if $z\in A_r$} \\ \\
R_rf_0R_r(z), & \hbox{if $z\in A_{r^2}\setminus A_r$.}
\end{array}
\right.$$
Since $f_0( S^1 )= S^1 $, one has $f_0(S_r)=S_r$ and $f_0=R_rf_0R_r$ on $S_r$. Thus $f_1$ is a homeomorphic extension of $f_0$ to $A_{r^2}$. Since $f_0$ is $k$-periodic, we easily check that $f_1$ is $k$-periodic by using $R_r^2=id$. Moreover, since $f_0$ is quasisymmetric, there is a constant $K\geq 1$ such that $f_0$ is $K$-quasiconformal on $\mbox{int}(A_r)$. Therefore $f_1$ is $K$-quasiconformal on $\mbox{int}(A_{r^2})$.

Let
$$f_2(z)=\left\{
\begin{array}{lll}
f_1(z), & \hbox{if $z\in A_{r^2}$} \\ \\
R_{r^2}f_0R_{r^2}(z), & \hbox{if $z\in A_{r^4}\setminus A_{r^2}$.}
\end{array}
\right.$$
By an analogous argument, $f_2$ is a $k$-periodic homeomorphic extension of $f_1$ to $A_{r^4}$ and is $K$-quasiconformal on $\mbox{int}(A_{r^4})$.

Inductively, we get a sequence of homeomorphisms $f_m:A_{r^{2^m}}\to A_{r^{2^m}}$ with the following properties: For every integer $m\geq 1$,
$f_m$ is a $k$-periodic homeomorphic extension of $f_{m-1}$ to $A_{r^{2^m}}$ and is $K$-quasiconformal on $\mbox{int}(A_{r^{2^m}})$.

Note that $\{A_{r^{2^m}}\}_{m=0}^\infty$ is an increasing sequence of annuli satisfying
$$\mathbb{D}=\{0\}\cup\bigcup_{m=0}^\infty A_{r^{2^m}},$$
where $\mathbb{D}$ is the closed unit disk. We may define $f_\infty:\mathbb{D}\to \mathbb{D}$ by
$$f_\infty(z)=\left\{
\begin{array}{lll}
f_m(z), & \hbox{if $z\in A_{r^{2^m}}$ for some $m\geq 0$,} \\ \\
0, & \hbox{if $z=0$.}
\end{array}
\right.$$
We easily check by using the above properties of the sequence $\{f_m\}_{m=1}^\infty$ that $f_\infty$ is a well-defined $k$-periodic homeomorphic extension of $f_0$ to $\mathbb{D}$ and is $K$-quasiconformal on $\mbox{int}(\mathbb{D})$.

Finally, let $R$ be the reflection through the unit circle $ S^1 $ and let
$$F(z)=\left\{
\begin{array}{lll}
f_\infty(z), & \hbox{if $z\in \mathbb{D}$} \\ \\
Rf_\infty R(z), & \hbox{if $z\in \mathbb{C}\setminus \mathbb{D}.$}
\end{array}
\right.$$
Then $F$ is a $k$-periodic $K$-quasiconformal extension of $f_0$, as desired.
\end{proof}

\begin{remark}
By the above argument, every $k$-periodic orientation-preserving quasisymmetric homeomorphism $f: S^1 \to  S^1 $ has a $k$-periodic quasiconformal extension $F:\mathbb{C}\to \mathbb{C}$ with $F(0)=0$.
\end{remark}

\noindent{\bf Proof of Theorem \ref{l2}.} Let $S:=D_0\setminus\cup_{j=1}^\infty D_j$ be a quasi-round carpet of measure zero in $\mathbb{C}$ and $f:S\to S$ be a $k$-periodic orientation-preserving quasisymmetric homeomorphism with $f(\partial D_0)=\partial D_0$. According to Bonk \cite[Proposition 5.1]{B}, $f$ has a quasiconformal extension $\widetilde{f}:D_0\to D_0$. Such an extension $\widetilde{f}$ is not necessarily $k$-periodic. We are going to modify it to create a $k$-periodic quasiconformal extension of $f$.

By the assumptions on $f$ and $S$, for every integer $j\geq1$ there is a factor $m_j$ of $k$ such that $\partial D_j, f(\partial D_j),\cdots, f^{m_j-1}(\partial D_j)$ are pairwise distinct peripheral circles of $S$ and that $f^{m_j}:\partial D_j\to\partial D_j$ is a $k/m_j$-periodic orientation-preserving quasisymmetric homeomorphism. Since $\partial D_j$ is a $c$-quasicircle, by Theorem \ref{l3} $f^{m_j}$ has a $k/m_j$-periodic quasiconformal extension $g_j:\overline{D}_j\to \overline{D}_j$. Moreover, since the constant $c$ is independent of $j$ and $1\leq m_j\leq k$, the quasiconformality constant of $g_j$ can be chosen independent of $j$. We call
$$\mathcal{O}_j:=\{\partial D_j, f(\partial D_j),\cdots, f^{m_j-1}(\partial D_j)\}$$ the orbit of the peripheral circle $\partial D_j$ under $f$. It is clear that two orbits are either equal or disjoint. Thus there is a subset $J$ of positive integers such that $\mathcal{O}_j$, $j\in J$, are pairwise disjoint,  with
$$
\bigcup_{j\in J}\mathcal{O}_j=\{\partial D_j: j\geq 1\}.
$$

Since $\widetilde{f}:D_0\to D_0$ is a quasiconformal extension of $f:S\to S$, we see that $D_j, \widetilde{f}(D_j),\cdots,\widetilde{f}^{\,m_j-1}(D_j)$ are components of $D_0\setminus S$ whose closures are pairwise disjoint, $\widetilde{f}^{\,m_j}(D_j)=D_j$,
and that
$$\bigcup_{j\in J}\bigcup_{n=0}^{m_j-1} \widetilde{f}\,^n(D_j)=D_0\setminus S.$$

For each $j\in J$  let
$$U_j=\bigcup_{n=0}^{m_j-1} \widetilde{f}\,^n(\overline{D}_j),\, V_j=\bigcup_{n=0}^{m_j-2} \widetilde{f}\,^n(\overline{D}_j),\, \mbox{ and }W_j=\bigcup_{n=1}^{m_j-1} \widetilde{f}\,^n(\overline{D}_j).$$
We have $\widetilde{f}(U_j)=U_j$ and $\widetilde{f}(V_j)=W_j$. To modify $\widetilde{f}|_{U_j}$ into a $k$-periodic map, we define
a new map $h_j:U_j\to U_j$ by
$$h_j(z)=\left\{
\begin{array}{ll}
\widetilde{f}(z), & \hbox{if $z\in V_j$} \\ \\
g_j\widetilde{f}^{\,-m_j+1}(z), & \hbox{if $z\in \widetilde{f}^{\, m_j-1}(\overline{D}_j)$.}
\end{array}
\right.$$
Since $\overline{D}_j$, $\widetilde{f}(\overline{D}_j)$, $\cdots$, and $\widetilde{f}^{\,m_j-1}(\overline{D}_j)$ are pairwise disjoint and $g_j(\overline{D}_j)=\overline{D}_j$, we see that $h_j$ is a well-defined homeomorphism.
By the definition, one has $h_j^{m_j}=g_j$ on $\overline{D}_j$. Since $g_j$ is $k/m_j$-periodic, we get that $h_j$ is $k$-periodic. Since $\widetilde{f}$ and $g_j$ are quasiconformal, we see that $h_j$ is quasiconformal on $\mbox{int}(U_j)$. Since $m_j\leq k$ and the quasiconformality constants of $\widetilde{f}$ and $g_j$ are independent of $j$, the quasiconformality constant of $h_j$ is independent of $j$.
Finally, we have
$$
h_j=f\,\mbox{ on }\,{\partial U_j},
$$
in fact, if $z\in\partial V_j$, we have $h_j(z)=\widetilde{f}(z)=f(z)$; if $z\in \widetilde{f}^{\, m_j-1}(\partial D_j)$, we have $\widetilde{f}^{\, -m_j+1}(z)\in \partial D_j$, and hence $$h_j(z)=g_j\widetilde{f}^{\, -m_j+1}(z)=f^{m_j}f^{-m_j+1}(z)=f(z).$$
To sum up, $h_j:U_j\to U_j$ is a $k$-periodic  orientation-preserving homeomorphic extension of $f|_{\partial U_j}$, which is $K$-quasiconformal on $\mbox{int}(U_j)$ for some constant $K$ independent of $j$.

Let $F: D_0\to D_0$ be defined to satisfy
$$F=h_j \mbox{ on }U_j \mbox{ for every }j\in J \mbox{ and } F=f\mbox{ on }  S.$$
Then $F$ is a well-defined $k$-periodic orientation-preserving homeomorphic extension of $f$ and is quasiconformal on $D_0\setminus S$. Since $S$ is of measure zero, we see that $F$ is quasiconformal on $\mbox{int}(D_0)$, and hence $F$ is quasisymmetric on $D_0$.

Now let $G:\mathbb{C}\to \mathbb{C}$ be a quasiconformal homeomorphism with $G(D_0)=\mathbb{D}$. Then $GFG^{-1}: \mathbb{D}\to\mathbb{D}$ is a $k$-periodic orientation-preserving quasisymmetric homeomorphism, which has a $k$-periodic quasiconformal extension to $\mathbb{C}$ by the Schwarz reflection principle. Thus $F: D_0\to D_0$ has such an extension. This completes the proof of Theorem \ref{l2}.
$\hfill\Box$

\section{The Proof of Theorem  \ref{c1}}

In this section we prove Theorem \ref{c1}. We start with some basic results on periodic topological self homeomorphisms of Jordan arcs and closed paths.

\begin{lemma}\label{2.1}
Let $h:[s,t]\to[s,t]$ be an increasing homeomorphism. If all points of $[s,t]$ are periodic points of $h$, then $h=id$.
\end{lemma}
\begin{proof} We argue by contradiction and assume that $h\neq id$. Then there exists a point $x\in[s,t]$ such that $h(x)\neq x$. Without loss of generality assume that $h(x)<x$. Then, since $h$ is increasing, one has $h^2(x)<h(x)<x$. Inductively, one has $h^n(x)<x$ for all positive integers $n$, contradicting the assumption that all points of $[s,t]$ are periodic points of $h$. Thus $h=id$.
\end{proof}

\begin{lemma}\label{2.2}
Let $\gamma:[0,1]\to \mathbb{C}$ be a Jordan arc and let $h:\gamma[0,1]\to \gamma[0,1]$ be a homeomorphism fixing $\gamma(0)$ and $\gamma(1)$. If all points of $\gamma[0,1]$ are periodic points of $h$, then $h=id$.
\end{lemma}
\begin{proof} By the assumption, the map $\gamma^{-1}h\gamma:[0,1]\to[0,1]$ is an increasing homeomorphism. By Lemma \ref{2.1}, $\gamma^{-1}h\gamma=id$ on $[0,1]$, so $h=id$ on $\gamma[0,1]$.
\end{proof}

\begin{lemma}\label{2.3}
Let $\gamma$ be a Jordan closed path in $\mathbb{C}$ and $h$ be an orientation-preserving self-homeomorphism of $\gamma$ with a fixed point. If all points of $\gamma$ are periodic points of $h$, then $h=id$.
\end{lemma}
\begin{proof} Let $z_0\in \gamma$ be a fixed point of $h$. For every $z\in \gamma\setminus \{z_0\}$ one has $h(z)\neq z_0$. Denote by $\gamma_{z_0z}$ the proper subarc of $\gamma$ from $z_0$ to $z$ clockwise. By the assumptions, we have $h(\gamma_{z_0z})=\gamma_{z_0h(z)}$, and hence $\gamma_{z_0h(z)}$ is a proper subarc of $\gamma$. Moreover, it holds either $\gamma_{z_0z}\subset\gamma_{z_0h(z)}$ or $\gamma_{z_0z}\supset\gamma_{z_0h(z)}$. Without loss of generality assume $\gamma_{z_0z}\subset\gamma_{z_0h(z)}$. Then $\gamma_{z_0h^n(z)}$ is a proper subarc of $\gamma$ and $$\gamma_{z_0z}\subset\gamma_{z_0h(z)}\subset\cdots\subset\gamma_{z_0h^n(z)}$$ for all integers $n\geq 1$. On the other hand, since $z$ is a periodic point of $h$, one has $\gamma_{z_0z}=\gamma_{z_0h^n(z)}$ for some integer $n\geq 1$. It then follows that $\gamma_{z_0z}=\gamma_{z_0h(z)}$, so $h(z)=z$. Thus $h=id$.
\end{proof}

The following result can be found in Ker\'{e}tj\'{a}rt\'{o} \cite{K}.

\begin{lemma}\label{2.4}
Let $G_1, G_2$ be Jordan domains in $\mathbb{C}$ with $G_1\cap G_2\neq\emptyset$. Then every component (maximal connected subset) of  $G_1\cap G_2$ is a Jordan domain.
\end{lemma}

The following topological rigidity result will play an important role in the proof of Theorem \ref{c1}.

\begin{lemma}\label{l1}
Let $G$ be a Jordan closed domain in $\mathbb{C}$. If $f:G\to G$ is a periodic orientation-preserving homeomorphism that has a fixed point in $\partial G$ then $f$ is the identity on $G$.
\end{lemma}

\begin{proof}
By the assumptions, $\partial G$ is a Jordan closed path in $\mathbb{C}$ and the restriction $h|_{\partial G}$ is an orientation-preserving self homeomorphism of $\partial G$ with a fixed point, and all points of $\partial G$ are periodic points of $h|_{\partial G}$. Thus $h=id$ on $\partial G$ by Lemma \ref{2.3}.

Fix two distinct points $u,v$ and a proper closed subarc $\alpha$ of endpoints $u,v$ in $\partial G$. To prove $f=id$ on $G$, it suffices to verify $h=id$ on $\gamma$ for every open Jordan arc $\gamma$ in $\mbox{int}(G)$ of endpoints $u,v$.

Let $\gamma$ be such an arc. Then $\alpha\cup \gamma$ is a Jordan closed path. Denote by $D$ the Jordan domain bounded by $\alpha\cup \gamma$ in $G$. Since $h=id$ on $\alpha$, we see that $h^m(D)$ is a Jordan domain bounded by $\alpha\cup h^m(\gamma)$ for every integer $m\geq 0$.
Let
$$
A=D\cap h(D)\cap\cdots\cap h^{k-1}(D).
$$
Then $A$ has a component $B$ whose boundary contains $\alpha$. By Lemma \ref{2.4}, $B$ is a Jordan domain. Let $\beta=(\partial B)\setminus\alpha$. Then $\beta$ is a Jordan arc of endpoints $u,v$ with
\begin{equation}\label{yy}
\beta\subset \gamma\cup h(\gamma)\cup\cdots\cup h^{k-1}(\gamma).
\end{equation}
Since $h$ is $k$-periodic, one has $h(A)=A$ by the definition of $A$, which implies that $h(B)$ is a component of $A$ whose boundary contains $\alpha$. Thus $h(B)=B$,
$h(\partial B)=\partial B$, and $h(\beta)=\beta$. Now, by Lemma \ref{2.2}, we have $h=id$ on $\beta$, which gives
\begin{equation}\label{be}
\beta\cap h^m(\gamma)=h^{-m}(\beta\cap h^m(\gamma))=\beta\cap \gamma
\end{equation}
for every integer $m\geq 0$.

The inclusion (\ref{yy}) and the equality (\ref{be}) verify $\beta=\beta\cap \gamma$, and hence $\beta\subset\gamma$. Since both $\gamma$ and $\beta$ are Jordan arcs in $\mbox{int}(G)$ of endpoints $u$ and $v$, the inclusion $\beta\subset\gamma$ is actually the equality $\beta=\gamma$.

Since $\beta=\gamma$ and $h=id$ on $\beta$, we have $h=id$ on $\gamma$, as desired.
\end{proof}

\noindent{\bf The proof of Theorem \ref{c1}.} Let $S:=D_0\setminus\cup_{k=1}^\infty D_k$ be a quasi-round carpet of measure zero in $\mathbb{C}$ and $f:S\to S$ be a periodic orientation-preserving quasisymmetric homeomorphism that has a fixed point in $\partial D_0$. By Theorem \ref{l2}, $f$ has a periodic quasiconformal extension $F:\mathbb{C}\to \mathbb{C}$. Thus, the restriction $F|_{D_0}$ is a periodic orientation-preserving self homeomorphism of $D_0$ with a fixed point on $\partial D_0$. By Lemma \ref{l1}, $F=id$ on $D_0$, so $f=id$ on $S$. $\hfill\Box$

\section{Rigidity of Square Carpets}

In this section we prove Theorems \ref{t1} and \ref{t1.6}.

\medskip

\noindent{\bf Proof of Theorem \ref{t1}.} Let
$$
R=[0,a]\times [0,1]\ \mbox{ and }\ K=(s,s+w)\times (t,t+h)\ \mbox{ with }\ \overline{K}\subset\mbox{int}(R),
$$
where $a,w,h>0$. Let
$$
S:=(R\setminus K)\setminus\bigcup_{k=1}^\infty Q_k
$$
be a square carpet of measure zero in the rectangle ring $R\setminus K$. Let $f: S\to S$ be a quasisymmetric homeomorphism fixing each of the four vertices of $R$. We first prove that there is an integer $n\geq 1$ such that
\begin{equation}\label{z1}
f^n(\partial K)=\partial K.
\end{equation}

The proof requires the concept of the modulus of a path family. Let $\Gamma$ be a family of rectifiable paths in $\mathbb{C}$. A Borel function $\rho: \mathbb{C}\to [0,\infty]$ is called admissible for $\Gamma$ if $$\int_{\gamma}\rho\ ds\geq 1$$
for each $\gamma\in\Gamma$. The conformal modulus $\mbox{mod}(\Gamma)$ of $\Gamma$ is defined by
$$
\mbox{mod}(\Gamma)=\inf_{\rho}\int\rho^2\ dxdy,
$$
where the infimum is taken over all admissible functions for $\Gamma$. An admissible function $\rho$ is called extremal for $\Gamma$ if
$$
\mbox{mod}(\Gamma)=\int\rho^2\ dxdy.
$$

Now we list key lemmas. According to Bonk \cite[Proposition 5.1]{B}, $f$ has a quasiconformal extension $F: \mathbb{C}\to \mathbb{C}$. Let such an extension $F$ be given.
For each $x\in [0,a]$ let $\gamma_x$ be the vertical segment joining points $(x,0)$ and $(x,1)$. Let
\begin{equation}\label{A}
A=\{x\in[0,a]:\, F\in\mbox{AC}(\gamma_x),\, \mathcal{L}(S\cap\gamma_x)=0\}
\end{equation}
and
$$\Gamma_A:=\{\gamma_x: x\in A\},$$ where $F\in\mbox{AC}(\gamma_x)$ means that $F$ is absolutely continuous on $\gamma_x$ and $\mathcal{L}$ denotes the $1$-dimensional Lebesgue measure.

\begin{lemma}\label{4.1}
$\mathcal{L}(A)=a$.
\end{lemma}

\begin{lemma}\label{4.2+}
$\mathcal{L}(P(F(\gamma)\setminus S))=1$ for every $\gamma\in \Gamma_A$,
where $P$ denotes the orthogonal projection to the vertical coordinative axis.
\end{lemma}

\begin{lemma}\label{4.2}
Let $\Gamma_A:=\{\gamma_x: x\in A\}$, where $A$ is defined as (\ref{A}). Then
$$
\mbox{\rm mod}(\Gamma_A)=a.
$$
Moreover, an admissible function $\rho$ is extremal for $\Gamma_A$ if and only if $\rho(z)=1$ for almost all $z$ in the rectangle $R$.
\end{lemma}
Lemma \ref{4.2} is well known; Lemmas \ref{4.1} and \ref{4.2+} are adapted from Bonk and Merenkov \cite{BM1}; they are applied to the proof of \cite[Theorem 1.4]{BM1}.

Let $\textbf{H}:=\{K\}\cup\{Q_k:k\geq 1\}$ be the set of all holes of the carpet $S$ in the rectangle $R$. Then
\begin{equation}\label{wx}
\{F(Q): Q\in \textbf{H}\}=\textbf{H}.
\end{equation}
For every $Q\in \textbf{H}$ we denote by $l(Q)$ the length of a vertical side of it. Then one has $l(K)=h$.

Since \cite[Theorem 1.4]{BM1} establishes $f=id$ on $S$ for the case $K$ is a square, we only need to prove (\ref{z1}) for the case $w\neq h$.

We argue by contradiction and assume $f^n(\partial K)\neq\partial K$ for all integers $n\geq 1$. This assumption and (\ref{wx}) imply that $F^n(K)$, $n\geq 1$, are different members of $\{Q_k:k\geq 1\}$ and that there exist two distinct positive integers $p, q$ such that $$F(K)=Q_p\ \mbox{ and }\ F(Q_q)=K.$$
Then, since $S$ is a carpet, one has $l(Q_k)\to 0$ as $k\to +\infty$, and hence
\begin{equation}\label{yishi1}
\inf_{n\geq 1} l(F^n(K))=0.
\end{equation}
To obtain a contradiction to (\ref{yishi1}), we are going to show
\begin{equation}\label{yishi}
\min\{w,h\}\leq \inf_{n\geq 1}l(F^n(K)).
\end{equation}
There are two cases.

\medskip

Case 1. $w>h$.

\medskip

We first show $h\leq l(F(K))$.
Let $\rho: \mathbb{C}\to [0,\infty]$ be defined by
$$\rho(z)=\left\{
\begin{array}{llll}
{l(Q_p)}/{h} & \hbox{if $z\in K$} \\ \\
{h}/{l(Q_q)} & \hbox{if $z\in Q_q$} \\ \\
{l(F(Q_k))}/{l(Q_k)} & \hbox{if $z\in Q_k,\ k\geq 1,\ k\neq q$}\\ \\
0 & \hbox{otherwise.}
\end{array}
\right.$$
We claim that $\rho$ is admissible for the path family $\Gamma_A$ in Lemma \ref{4.2}.

In fact, for each $\gamma\in \Gamma_A$ one has $\mathcal{L}(\gamma\cap S)=0$. By this equality, (\ref{wx}), and Lemma \ref{4.2+}, we then get
\begin{eqnarray*}
\int_\gamma\rho\ ds
&=&l(Q_p)\frac{\mathcal{L}(\gamma\cap K)}{h}+h\frac{\mathcal{L}(\gamma\cap Q_q)}{l(Q_q)}+\sum_{k\geq 1,\, k\neq q}l(F(Q_k))\frac{\mathcal{L}(\gamma\cap Q_k)}{l(Q_k)}\\
&=& \sum_{Q\in \textbf{H},\, Q\cap\gamma\neq\emptyset}l(F(Q))\\
&=&\sum_{Q\in \textbf{H},\,F(\gamma)\cap Q\neq\emptyset}l(Q)\\
&\geq& \mathcal{H}^1(P(F(\gamma)\setminus S))=1.
\end{eqnarray*}
This proves that $\rho$ is admissible for $\Gamma_A$.

Since $\mbox{\rm mod}(\Gamma_A)=a$ by Lemma \ref{4.2}, we have
\begin{eqnarray*}
a\leq\int_{R}\rho^2\ dxdy &=& \frac{l(Q_p)^2}{h^2}wh+h^2+\sum_{k\geq 1, k\neq q}l(F(Q_k))^2 \\
&=&\frac{w}{h}l(Q_p)^2+h^2+a-l(Q_p)^2-wh,
\end{eqnarray*}
which implies
$$(\frac{w}{h}-1)h^2\leq(\frac{w}{h}-1)l(Q_p)^2.$$
Since $w>h$ is assumed, we get
$$
h\leq l(Q_p)=l(F(K)).
$$

Now we prove that $h\leq l(F(K))$. Let $n\geq 1$ be an integer. Replacing $F$ with $F^n$ in the above argument, we get
$$
h\leq l(F^n(K)).
$$

\medskip

Case 2. $w<h$.

\medskip

In this case, consider the family of horizontal segments in the rectangle $R$ and let
$$\Gamma:=\{\gamma_y:\, y\in[0,1],\, F\in\mbox{AC}(\gamma_y),\, \mathcal{L}(S\cap\gamma_y)=0\}.$$
One has $\mbox{mod}(\Gamma)=1/a$.
Let $\rho: \mathbb{C}\to [0,\infty]$ be defined by
$$\rho(z)=\left\{
\begin{array}{llll}
{l(Q_p)}/w & \hbox{if $z\in K$} \\ \\
w/{l(Q_q)} & \hbox{if $z\in Q_q$} \\ \\
{l(F(Q_k))}/{l(Q_k)} & \hbox{if $z\in Q_k,\ k\geq 1,\ k\neq q$}\\ \\
0 & \hbox{otherwise.}
\end{array}
\right.$$
We may check that $\rho/a$ is admissible for $\Gamma$. Since $\mbox{\rm mod}(\Gamma)=1/a$, we have
\begin{eqnarray*}
a\leq\int_{R}\rho^2\ dxdy &=& \frac{l(Q_p)^2}{w^2}wh+w^2+\sum_{k\geq 1, k\neq q}l(F(Q_k))^2 \\
&=&\frac{h}{w}l(Q_p)^2+w^2+a-l(Q_p)^2-wh.
\end{eqnarray*}
Now, arguing as we just did in Case 1, we get
$$
w\leq l(F^n(K))
$$
for every integer $n\geq 1$. This proves (\ref{yishi}).

Since (\ref{yishi}) contradicts (\ref{yishi1}), we conclude that (\ref{z1}) is true. Now, by (\ref{z1}) and  Lemma \ref{1st lemma}, we see that $f^n=id$ on $S$ for some integer $n\geq 1$.
Therefore, every homeomorphism $f:S\to S$ in Theorem \ref{t1} is actually a periodic orientation-preserving quasisymmetric homeomorphism with a fixed point in $\partial R$. It then follows from Theorem \ref{c1} that $f=id$ on $S$. $\hfill\Box$

\medskip

The proof of Theorem \ref{t1.6} follows the same structure as Theorem \ref{t1}. The definition of the modulus of a path family in $\mathbb{C}^*$ will be used. Let $$ds=\frac{|dz|}{|z|}\,\mbox{ and }\, d\sigma=\frac{dxdy}{|z|^2}$$ be the line element and the area element of $\mathbb{C}^*$. Let $\Gamma$ be a family of rectifiable paths in $\mathbb{C}^*$. We say that a Borel function $\rho: \mathbb{C}^*\to [0,\infty]$ is admissible for $\Gamma$ if $$\int_{\gamma}\rho(z)\ \frac{|dz|}{|z|}\geq 1$$
for each $\gamma\in\Gamma$. The modulus  of $\Gamma$ is defined by
\begin{equation}\label{md}
\mbox{mod}(\Gamma)=\inf_{\rho}\int\rho^2\ d\sigma,
\end{equation}
where the infimum is taken over all admissible functions for $\Gamma$.

\medskip

\noindent{\bf Proof of Theorem \ref{t1.6}.} Let $A:=\{1\leq|z|\leq r\}$ be a finite $\mathbb{C}^*$-cylinder and $$K=\{te^{i\theta}: a<t<b,\ \alpha<\theta<\beta\}$$
be a $\mathbb{C}^*$-rectangle with $1<a<b<r$ and $0<\beta-\alpha<2\pi$. Let $$h=\log\frac{b}{a}\ \mbox{ and }\ w=\beta-\alpha,$$ which are respectively the radial and circular directional side-lengths of $K$. Let
$$
S:=(A\setminus K)\setminus\bigcup_{k=1}^\infty Q_k
$$
be a $\mathbb{C}^*$-square carpet of measure zero in $A\setminus K$.
Let $f:S\to S$ be an orientation-preserving quasisymmetric homeomorphism that maps the inner and outer peripheral circles of $S$ onto themselves, respectively. We are going to show $f^n=id$ on $S$ for some integer $n\geq1$.

According to Bonk \cite[Proposition 5.1]{B}, $f$ has a quasiconformal extension $F: \mathbb{C}\to \mathbb{C}$.
Let $\textbf{H}:=\{K\}\cup\{Q_k:k\geq 1\}$ be the set of all holes of the carpet $S$ in the finite $\mathbb{C}^*$-cylinder $A$. Then
\begin{equation}\label{wx}
\{F(Q): Q\in \textbf{H}\}=\textbf{H}.
\end{equation}
For every $Q\in \textbf{H}$ we denote by $l^*(Q)$ the $\mathbb{C}^*$-length of a radial side of it. Then one has $l^*(K)=h$.
For every $\theta\in[0,2\pi)$ denote by $\gamma_\theta$ the radial segment joining $e^{i\theta}$ and $re^{i\theta}$. Then $\gamma_\theta$ is of $\mathbb{C}^*$-length $\log r$. Let
$$
\Gamma=\{\gamma_\theta:\, \theta\in[0,2\pi),\, F\in\mbox{AC}(\gamma_\theta),\, \mathcal{L}(S\cap \gamma_\theta)=0\}.
$$
By the definition (\ref{md}), we easily get
\begin{equation}\label{mof}
\mbox{\rm mod}(\Gamma)=\frac{2\pi}{\log r}.
\end{equation}
We claim that there is an integer $n\geq 1$ such that
\begin{equation}\label{zx}
f^n(\partial K)=\partial K.
\end{equation}
Since \cite[Theorem 1.5]{BM1} verifies (\ref{zx}) in the case $K$ is a $\mathbb{C}^*$-square.
We only need to prove it for the case $w\neq h$. We argue by contradiction and assume $f^n(\partial K)\neq\partial K$ for all integers $n\geq 1$. Then $F^n(K)$, $n\geq 1\}$ are different members of $\{Q_k:k\geq 1\}$ and there exist two distinct positive integers $p, q$ such that $$F(K)=Q_p\ \mbox{ and }\ F(Q_q)=K.$$ Since $S$ is a carpet, one has $l^*(Q_k)\to 0$ as $k\to +\infty$. Therefore
\begin{equation}\label{yishi1*}
\inf_{k\geq 1} l^*(F^k(K))=0.
\end{equation}
Next we show
\begin{equation}\label{yishi*}
\min\{w,h\}\leq \inf_{k\geq 1}l^*(F^k(K)).
\end{equation}
There are two cases.

\medskip

Case 1. $w>h$.

\medskip

We first show $h\leq l^*(F(K))$.
Let $\rho: \mathbb{C}^*\to [0,\infty]$ be defined by
$$\rho(z)=\left\{
\begin{array}{llll}
{l^*(Q_p)}/h& \hbox{if $z\in K$} \\ \\
h/{l^*(Q_q)} & \hbox{if $z\in Q_q$} \\ \\
{l^*(F(Q_k))}/{l^*(Q_k)} & \hbox{if $z\in Q_k,\ k\geq 1,\ k\neq q$}\\ \\
0 & \hbox{otherwise.}
\end{array}
\right.$$
We easily check that $\rho/\log r$ is admissible for the path family $\Gamma$. By (\ref{mof}) we have
$$2\pi\log r\leq\int_A\rho^2\ d\sigma=\frac{l^*(Q_p)^2}{h^2}wh+h^2+\sum_{k\geq 1, k\neq q}l^*(F(Q_k))^2,$$
which together with $\sigma(A)=2\pi\log r$, $\sigma(Q_p)=l^*(Q_p)^2$, and $\sigma(K)=wh$ implies
$$2\pi\log r\leq
\frac{w}{h}l^*(Q_p)^2+h^2+2\pi\log r-l^*(Q_p)^2-wh,$$
and hence
$$(\frac{w}{h}-1)h^2\leq(\frac{w}{h}-1)l^*(Q_p)^2.$$
Since $w>h$ is assumed, we thus get
$$
h\leq l^*(Q_p)=l^*(F(K)).
$$
Similarly, we have $h\leq l^*(F^n(K))$ for all integers $n\geq 1$.

\medskip

Case 2. $w<h$.

\medskip

Let $\vartheta_t$ denote the circle $|z|=t$ and let $$\Gamma=\{\vartheta_t:\, t\in[1,r],\, F\in\mbox{AC}(\vartheta_t),\, \mathcal{L}(S\cap \vartheta_t)=0\}.$$ One has $$\mbox{mod}(\Gamma)=\frac{\log r}{2\pi}.$$ By constructing a suitable admissible function for $\Gamma$ and arguing as we did in Case 1, we get
$w\leq l^*(F^n(K))$ for all integers $n\geq 1$. This proves (\ref{yishi*}).

It is clear that (\ref{yishi*}) contradicts (\ref{yishi1*}), and hence (\ref{zx}) is ture. Now, by (\ref{zx}) and Lemma \ref{2cd lemma}, we see that $f^n=id$ on $S$ for some integer $n\geq 1$. This completes the proof.  $\hfill\Box$


\end{document}